\documentclass[twosided,a4paper,11pt]{amsart}

\usepackage{setspace}
\usepackage[T1]{fontenc}
\usepackage[dvips]{graphicx}
\usepackage{epsfig}
\usepackage{amsmath,amssymb,amscd,amsthm}
\usepackage{subfigure}
\usepackage[latin1]{inputenc}
\usepackage{latexsym}
\usepackage{graphics}
\usepackage{colortbl}
\usepackage{color}
\usepackage{ae}
\usepackage{enumitem}
\usepackage[all]{xy}

\newtheorem{cor}{Corollary}[section]
\newtheorem{theorem}[cor]{Theorem}
\newtheorem*{theorem*}{Theorem}
\newtheorem{prop}[cor]{Proposition}
\newtheorem{lemma}[cor]{Lemma}

\newtheorem{fact}[cor]{Fact}
\theoremstyle{definition}
\newtheorem{defi}[cor]{Definition}
\theoremstyle{remark}
\newtheorem{remark}[cor]{Remark}

\newcommand{\C}{{\mathbb C}}

\newcommand{\R}{{\mathbb R}}

\newcommand{\Z}{{\mathbb Z}}

\newcommand{\diff}{\mathrm{Diff}}
\newcommand{\symp}{\mathrm{Symp}}
\newcommand{\ham}{\mathrm{Ham}}
\newcommand{\Hom}{\mathrm{Hom}}
\newcommand{\boh}{\mathcal{C}}
\newcommand{\dR}{\mathrm{dR}}

\newcommand{\flux}{\mathrm{Flux}}

\newcommand{\Hyp}{\mathbb{H}}
\newcommand{\AdS}{\mathbb{A}\mathrm{d}\mathbb{S}}

\newcommand{\psl}{\mathfrak{sl}}

\newcommand{\PSL}{\mathrm{PSL}}
\newcommand{\PSLR}{\mathrm{PSL}_2\R}

\newcommand{\ML}{{\mathrm{M}\!\mathrm{L}}}
\newcommand{\lab}{\mathrm{L}}

\newcommand{\tr}{\mbox{\rm tr}\,}

\newcommand{\grad}{\operatorname{grad}}
\newcommand{\isom}{\mathrm{Isom}}

\newcommand{\RP}{\R \mathrm{P}}

\begin{document}

\setcounter{secnumdepth}{3}
\setcounter{tocdepth}{2}

\title[Flux homomorphism and Anti-de Sitter geometry]{The flux homomorphism on closed\\hyperbolic surfaces and Anti-de Sitter three-dimensional geometry}

\author[Andrea Seppi]{Andrea Seppi}
\address{Andrea Seppi: Dipartimento di Matematica ``Felice Casorati", Universit\`{a} degli Studi di Pavia, Via Ferrata 5, 27100, Pavia, Italy.} \email{andrea.seppi01@ateneopv.it}


\begin{abstract}
Given a smooth spacelike surface $\Sigma$ of negative curvature in Anti-de Sitter space of dimension 3, invariant by a representation $\rho:\pi_1(S)\to\PSLR\times\PSLR$ where $S$ is a closed oriented surface of genus $\geq 2$, a canonical construction associates to $\Sigma$ a diffeomorphism $\phi_\Sigma$ of $S$. It turns out that $\phi_\Sigma$ is a symplectomorphism for the area forms of the two hyperbolic metrics $h$ and $h'$ on $S$ induced by the action of $\rho$ on $\Hyp^2\times\Hyp^2$. Using an algebraic construction related to the flux homomorphism, we give a new proof of the fact that $\phi_\Sigma$ is the composition of a Hamiltonian symplectomorphism of $(S,h)$ and the unique minimal Lagrangian diffeomorphism from $(S,h)$ to $(S,h')$.
\end{abstract}

\maketitle

\section{Introduction}

Anti-de Sitter space is a real Lorentzian three-manifold of constant negative sectional curvature, which can be defined as the Lie group $\PSLR$ endowed with the Lorentzian metric induced by the Killing form. In recent times, since the groundbreaking paper \cite{mess} of Mess, its study has spread widely, mostly motivated by the relations between Anti-de Sitter space (which is denoted $\AdS^3$) and Teichm\"uller theory of hyperbolic surfaces --- and the present paper lies in this research direction as well.

Let us explain one of the instances of this relation. A remarkable construction permits to associate to a smooth spacelike surface $\Sigma$ (topologically a disc) in $\AdS^3$ a submanifold $\Lambda_\Sigma$ of $\Hyp^2\times \Hyp^2$, where $\Hyp^2$ is the hyperbolic plane. This is essentially due to the fact that the space of timelike lines of $\AdS^3$ is identified to $\Hyp^2\times \Hyp^2$. Hence the submanifold $\Lambda_\Sigma$ is the set of timelike lines orthogonal to $\Sigma$ --- roughly speaking, the analogue of the Gauss map in this context. 

It turns out, as observed in \cite{barbotkleinian} and \cite{bonsepequivariant}, that $\Lambda_\Sigma$ is always a \emph{Lagrangian} submanifold, for the symplectic structure which makes $\Hyp^2\times \bar{\Hyp}^2$ a Kähler manifold, namely $\Omega=(\pi_l^*\Omega_{\Hyp^2}-\pi_r^*\Omega_{\Hyp^2})$, where $\Omega_{\Hyp^2}$ is the area form of $\Hyp^2$, $\pi_l,\pi_r$ are the projections to each factor of $\Hyp^2\times\bar{\Hyp}^2$, and $\bar{\Hyp}^2$ denotes $\Hyp^2$ endowed with the opposite orientation. 
There are two classes of particular cases of this construction:
\begin{itemize}
\item If $\Sigma$ has negative curvature, then $\Lambda_\Sigma$ is the graph of a symplectomorphism $\phi_\Sigma$ of $(\Hyp^2,\Omega_{\Hyp^2})$ --- this is the prototype of Lagrangian submanifolds of $\Hyp^2\times \Hyp^2$. This case was first considered in  \cite{Schlenker-Krasnov}. 
\item If $\Sigma=\Sigma_0$ has vanishing mean curvature (also called \emph{maximal}, which implies negative curvature), then $\phi_{\Sigma_0}$ is \emph{minimal Lagrangian}, that is, it is a symplectomorphism and its graph is a minimal surface in the Riemannian product $\Hyp^2\times \Hyp^2$. See \cite{bon_schl,torralbo,bbzads,seppimaximal,seppiminimal} for results in this direction.
\end{itemize}
From the results of \cite{barbotkleinian} and \cite{ksurfaces}, it follows that being Lagrangian is essentially the only obstruction to inverting this construction, namely to realizing a submanifold of $\Hyp^2\times\Hyp^2$ as the image of the Gauss map of a spacelike surface in $\AdS^3$.

However, the situation is extremely more interesting when we consider cocompact actions. More precisely, let $S$ be a closed oriented surface of negative Euler characteristic. Suppose $\Sigma$ is invariant by an action of the group $\pi_1(S)$, which preserves the orientation and the time-orientation of $\AdS^3$. This produces therefore a representation $\rho$ of $\pi_1(S)$ in the isometry group $\isom(\AdS^3)$, which (from the definition of $\AdS^3$ as $\PSLR$ endowed with the bi-invariant metric) is naturally isomorphic to $\PSLR\times\PSLR$. By the theory developed in \cite{mess}, it turns out that $\rho=(\rho_l,\rho_r)$ where $\rho_l$ and $\rho_r$ are Fuchsian representations --- that is,  $\rho_l(\pi_1(S))$ and $\rho_r(\pi_1(S))$
act freely and properly discontinuously on $\Hyp^2$, with quotient a hyperbolic surface homeomorphic to $S$. 

Let us identify the two quotient hyperbolic surfaces $\Hyp^2/\rho_l(\pi_1(S))$ and $\Hyp^2/\rho_r(\pi_1(S))$ by $(S,h_l)$ and $(S,h_r)$ respectively, where $h_l$ and $h_r$ are hyperbolic metrics on $S$. Then the Lagrangian submanifold $\Lambda_\Sigma$ of $\Hyp^2\times\Hyp^2$ descends to a Lagrangian submanifold in the quotient $(S\times S,h_l\oplus h_r)$. In particular, if the cocompact surface $\Sigma$ has negative curvature, then $\phi_\Sigma$ induces a symplectomorphism, which we denote again by $\phi_\Sigma:(S,\Omega_{h_l})\to (S,\Omega_{h_r})$. On the other hand, again by the work of Mess, whenever one picks two Fuchsian representations $\rho_l$ and $\rho_r$, there is abundance of embedded surfaces $\Sigma\subset \AdS^3$ on which $\pi_1(S)$ acts freely and properly discontinuously by means of the representation $\rho=(\rho_l,\rho_r):\pi_1(S)\to\PSLR\times\PSLR$. For instance, there always exists a (unique) invariant maximal surface $\Sigma_0$ (i.e. of vanishing mean curvature, see \cite{bbzads}), which induces the minimal Lagrangian diffeomorphism isotopic to the identity:
$$\phi_{\Sigma_0}=\phi_\ML:(S,h_l)\to (S,h_r)~.$$
It was already known from the results of \cite{labourieCP} and \cite{Schoenharmonic} that such a minimal Lagrangian diffeomorphism $\phi_\ML$ exists and is unique, for any two closed hyperbolic surfaces $(S,h_l)$ and $(S,h_r)$.

Hence it is a natural question to characterize the symplectomorphisms 
$$\phi_{\Sigma}=:(S,\Omega_{h_l})\to (S,\Omega_{h_r})~,$$
which arise as the symplectomorphism associated to a cocompact surface $\Sigma$ in $\AdS^3$.
The main result of this paper is a new proof of the following:

\begin{theorem*} 
Let $\rho_l,\rho_r:\pi_1(S)\to\PSLR$ be Fuchsian representations and let $\Sigma\subset \AdS^3$ be a smooth, embedded spacelike surface invariant for the representation $\rho=(\rho_l,\rho_r):\pi_1(S)\to\isom_0(\AdS^3)$, whose curvature is negative. 
Then 
$$\phi_\Sigma=\phi_{\ML}\circ \psi~,$$
where 
\begin{itemize}
\item $\phi_\Sigma:(S,\Omega_{h_l})\to (S,\Omega_{h_r})$ is the diffeomorphism associated to $\Sigma$;
\item  $\phi_{\ML}:(S,{h_l})\to (S,{h_r})$ is the unique minimal Lagrangian diffeomorphism isotopic to the identity;
\item $\psi$ is a Hamiltonian symplectomorphism, for the area form $\Omega_{h_l}$.
\end{itemize}
\end{theorem*}

Given a symplectic manifold $(M,\Omega)$, a symplectomorphism $\psi:(M,\Omega)\to (M,\Omega)$ is Hamiltonian if there exists an isotopy $\psi_t$ with $\psi_0=\mathrm{id}$ and $\psi_1=\psi$ and a smooth family of functions $H_t:M\to\R$ such that the generating vector field $X_t$ of the isotopy $\psi_t$ is the symplectic gradient of $H_t$ for every $t$. It turns out that Hamiltonian symplectomorphisms form a group, denoted $\ham(M,\Omega)$.

The above theorem is actually a consequence of the main result in \cite{bonsepequivariant}, where it was proved that any submanifold $\Lambda_\Sigma$ associated to an invariant surface $\Sigma$ is in the same $\ham(S\times S,\pi_l^*\Omega_{h_l}-\pi_r^*\Omega_{h_r})$-orbit of the submanifold $graph(\phi_\ML)$. This implies that, when $\Sigma$ has negative curvature, and therefore $\Lambda_\Sigma$ is the graph of the symplectomorphism $\phi_\Sigma:(S,\Omega_{h_l})\to (S,\Omega_{h_r})$, then $\phi_\Sigma$ and $\phi_\ML$ differ by a Hamiltonian symplectomorphism of $(S,\Omega_{h_l})$. 

However, we provide here a new proof, when $\Sigma$ has negative curvature. This relies on the construction of a map
$\boh_{h,h'}$, defined on $\symp_0(S,\Omega_h,\Omega_{h'})$ with values in the quotient of the de Rham cohomology group $H^1_\dR(S,\R)\cong \Hom(\pi_1(S),\R)$ by the discrete subgroup of 1-forms with periods integer multiples of $2\pi$, which is identified to $\Hom(\pi_1(S),2\pi\Z)$. This map is given by a rather algebraic construction, which involves differential-geometric invariants of the two hyperbolic surfaces $h$ and $h'$. Hence \emph{a priori} it depends on the choice of the two hyperbolic metrics. However, we then show that, when $h=h'$, this map coincides with the so-called \emph{flux homomorphism} (see \cite{MR0350776}, \cite{MR490874} and \cite[Chapter 6]{MR1698616}):
$$\flux:\symp_0(S,\Omega)\to H^1_\dR(S,\R)~,$$
post-composed with the projection to the quotient of $H^1_\dR(S,\R)$ by the subspace $\Hom(\pi_1(S),2\pi\Z)$. In particular, when $h=h'$, then the map $\boh_{h,h}$ turns out to be a homomorphism (this is proved directly, and is an important step of the proof), and only depends on the area form of $h$.
The proof of the main theorem then follows from interpreting such map $\boh_{h,h'}$ in terms of Anti-de Sitter geometry, and showing essentially  that it vanishes when the symplectomorphism $\phi_\Sigma$ is associated to an invariant spacelike surface $\Sigma$. One then obtains our main result as a consequence of the classical fact that the group of Hamiltonian symplectomorphisms $\ham(M,\Omega)$ is exactly the kernel of the flux homomorphism $\flux$, applied to the symplectic manifold $(M,\Omega)=(S,\Omega_{h_l})$. 

The paper then terminates with some remarks about the converse statement of our main theorem. Namely, given $\phi$ of the form $\phi_{\ML}\circ \psi$, for $\psi\in \ham(S,\Omega_{h_l})$, one can construct an invariant surface $\Sigma$ in $\AdS^3$, which is the candidate to be a surface such that $\phi_\Sigma=\phi$. This is indeed the case if $\Sigma$ is embedded. However, such surface $\Sigma$ will in general develop singularities. Although this generality is not considered here, in \cite{bonsepequivariant} it was showed that one can always construct a smooth lift of $\Sigma$ to the unit tangent bundle of the universal cover of $\AdS^3$, thus partially reversing the implication proved here.

\subsection*{Acknowledgements}
I would like to thank Thierry Barbot, Francesco Bonsante and Alessandro Ghigi for many invaluable conversations.



\section{Algebraic construction for the flux homomorphism} \label{sec algebraic}

In this paper, $S$ will denote a closed oriented surface of Euler characteristic $\chi(S)<0$. Moreover, we will usually fix two hyperbolic metrics (i.e. Riemannian metrics of constant curvature $-1$) $h$ and $h'$ on $S$.

In the first section, we will recall the classical definition of flux homomorphism, for a symplectic surface $(S,\Omega)$, from the point of view of symplectic geometry. Then we will give the definition of the alternative map $\boh_{h,h'}$ which \emph{a priori} depends on the choice of hyperbolic metrics $h$ and $h'$ on $S$. We will study some of its properties, for instance the fact that it is a homomorphism in the case $h=h'$, and we will show that the two homomorphisms essentially coincide in this case. 

\subsection*{The flux homomorphism}

In the definitions below, we will only suppose that $S$ is a closed oriented surface endowed with a symplectic form $\Omega$. Later, we will consider $\Omega=\Omega_h$ as the area form induced by a hyperbolic metric $h$.
Let us denote by $\symp_0(S,\Omega)$ the group of symplectomorphisms $\psi:(S,\Omega)\to(S,\Omega)$ isotopic to the identity.

Let us recall the definition of Hamiltonian symplectomorphisms (see \cite[Chapter 6]{MR1698616} as a reference):

\begin{defi}
Let $(S,\Omega)$ be a closed symplectic manifold. Then a symplectomorphism $\psi:(S,\Omega)\to(S,\Omega)$ is \emph{Hamiltonian} if there exists a smooth isotopy 
$$\psi_\bullet:[0,1]\to \symp_0(S,\Omega)~,$$ 
with $\psi_0=\mathrm{id}$ and $\psi_1=\psi$, and a smooth map
$$H_\bullet:[0,1]\times S\to \R~,$$
such that (if we denote $H_t(p)=H(t,p)$):
$$\Omega(X_t,\cdot)=dH_t~,$$
where $X_t$ is the generating vector field of the isotopy $\psi_t$, namely:
$$\frac{d}{dt}\psi_t=X_t\circ \psi_t~.$$
\end{defi}

It turns out that the space of Hamiltonian symplectomorphisms is a group, which we denote by $\ham(S,\Omega)$. Let us now define the flux homomorphism
$$\flux:\symp_0(S,\Omega)\to H^1_\dR(S,\R)~.$$

\begin{defi}
Given a symplectomorphism $\psi:(S,\Omega)\to (S,\Omega)$ isotopic to the identity, take an isotopy $\psi_t$ with $\psi_0=\mathrm{id}$ and $\psi_1=\psi$, and define
$$\flux(\psi)=\int_0^1 [\Omega(X_t,\cdot)]dt\in H^1_\dR(S,\R)~,$$
where $X_t$ is the generating vector field  of the isotopy $\psi_t$, namely:
$$\frac{d}{dt}\psi_t=X_t\circ \psi_t~.$$
\end{defi}

It turns out that, if there exists a homotopy in $\symp_0(S,\Omega)$ between the two paths 
$$\psi_\bullet,\psi'_\bullet:[0,1]\to\symp_0(S,\Omega)$$
 with fixed endpoints, then the value of $\flux(\psi)$ does not change. Since the group $\symp_0(S,\Omega)$ is simply connected, see \cite[Section 7.2]{polterovich},  $\flux$ is well-defined on $\symp_0(S,\Omega)$. In fact, by the Moser isotopy argument (see \cite{mcduffsurvey}), the inclusion of $\symp(S,\Omega)$ into $\diff(S)$ is a homotopy equivalence, and $\diff_0(S)$ is contractible by \cite{earleeelsdiffo}.

The flux homomorphism provides the following characterization of Hamiltonian symplectomorphisms:

\begin{theorem} \label{thm flux hamiltonian}
The sequence
\[
\xymatrix{
1\ar[r] &\ham(S,\Omega)\ar[r]^{\!\!\!\!\!\!i} &\symp_0(S,\Omega)\ar[r]^\flux & H^1_\dR(S,\R)\ar[r]  & 1
}
\]
is a short exact sequence of groups. In particular, a symplectomorphisms $\psi:(S,\Omega)\to(S,\Omega)$ is Hamiltonian if and only if $\flux(\psi)=0$.  
\end{theorem}


\subsection*{An alternative homomorphism}

Let us now fix two hyperbolic metrics $h$ and $h'$ on $S$, let $\Omega_h$ and $\Omega_{h'}$ be the area forms induced by $h$ and $h'$, and let $\psi:(S,\Omega_h)\to(S,\Omega_{h'})$ be a symplectomorphism. 

We denote by $\isom(TS,\psi^*h',h)$ the 
subbundle of the bundle $\mathrm{End}(T S)$, whose fibers over $x\in S$ are linear orientation-preserving automorphisms of $T_x S$ which are isometries between the metrics $\psi^*h'$ and $h$. 




Now, given a symplectomorphism $\psi$ and a smooth section $b$ of the bundle $\isom(TS,\psi^*h',h)$, we will define a smooth 1-form 
$$\eta_{\psi,b}\in Z^1_{\dR}(S,\R)~.$$
The 1-form $\eta_{\psi,b}$ will \emph{a priori} depend on the hyperbolic metrics $h,h'$, even when $h$ and $h'$ coincide.
 
We start by defining the 1-form locally. Let $\{v_1,v_2\}$ be an oriented orthonormal frame for $h$, on an open subset $U$ of $S$, and let $\omega$ be the associated connection form for the Levi-Civita connection $\nabla$ of $h$, which is defined by:
$$\nabla_v v_1=\omega(v)v_2\,.$$
Analogously, since $\psi^*h'=h(b\cdot,b\cdot)$, let $\omega'$ be the connection form associated to the oriented $\psi^*h'$-orthonormal frame $\{v'_1,v'_2\}:=\{b^{-1}v_1,b^{-1}v_2\}$ and to the Levi-Civita connection $\nabla'$ of $\psi^*h'$, that is:
$$\nabla'_v (b^{-1}v_1)=\omega'(v)(b^{-1}v_2)\,.$$ 
Then define
$$\eta_{\psi,b}=\omega'-\omega\,.$$

The following lemma shows that this definition does not depend on the choice of the orthonormal frame, and therefore this local definition provides a well-defined global 1-form on $S$.
\begin{lemma} \label{eta well defined}
The definition of $\eta_{\psi,b}$ does not depend on the choice of the oriented orthonormal frame $v_1,v_2$.
\end{lemma}
\begin{proof}
Let $\{\widehat v_1,\widehat v_2\}$ be another oriented orthonormal frame for $h$, on the open set $U$, which we can assume simply connected. Then there exists a smooth function $\theta:U\to\R$ (unique up to multiples of $2\pi$) such that
$$\widehat v_i=R_\theta (v_i)\,,$$
where $R_\theta\in\Gamma^\infty(U,\isom(TS,h))$ is the section given, at every point $x\in U$, by  counterclockwise rotation fixing $x$ of angle $\theta$ for the metric $h$. By a direct computation,
$$\nabla_v\widehat v_1=R_\theta \nabla_v v_1+d\theta(v)J_hR_\theta v_1\,,$$
where $J_h=R_{\pi/2}$ is the almost-complex structure of $h$. Hence
$$\widehat \omega(v)=h(\nabla_v \widehat v_1,\widehat v_2)=h(R_\theta \nabla_v v_1,R_\theta v_2)+d\theta(v) h(R_\theta v_2,R_\theta v_2)=\omega(v)+d\theta(v)\,.$$
Observing that, since $\psi^*h'=h(b\cdot,b\cdot)$, the rotation $R'_\theta$ for the metric $\psi^*h'$ coincides with $b^{-1}R_\theta b$, the orthonormal frame $\{\widehat v_1',\widehat v_2'\}$ is $\{R'_\theta v_1',R'_\theta v_2'\}$, and thus by the same computation:
$$\widehat \omega'(v)=\omega'(v)+d\theta(v)\,.$$
Therefore $\widehat \omega'-\widehat \omega=\omega'-\omega$ on $U$, and this concludes the proof.
\end{proof}

Hence we defined a 1-form $\eta_{\psi,b}$ associated to the pair $(\psi,b)$. Let us show that this form is closed:

\begin{lemma}
The 1-form $\eta_{\psi,b}$ is closed.
\end{lemma}
\begin{proof}
By the well-known formula for the curvature form, one has 
$$d\omega=-\Omega_h\,,$$
and analogously
$$d\omega'=-\Omega_{\psi^*h'}=-\Omega_h\,,$$
since $\psi$ is a symplectomorphism.
Therefore $d\eta_{\psi,b}=d\omega'-d\omega=0$.
\end{proof}

Unfortunately, the cohomology class of $\eta_{\psi,b}$ defined in this way is not independent of the choice of $b$, once the symplectomorphism $\psi$ isotopic to the identity is fixed. However, it can only differ in a controlled way, namely by cohomology classes with integer periods:


\begin{lemma} \label{lemma integer coefficients}
Let $b$ and $\widehat b$ be smooth sections of $\isom(TS,\psi^*h',h)$. Then 
$$[\eta_{\psi,b}]-[\eta_{\psi,\widehat b}]\in H^1_\dR(S,2\pi\Z)~.$$
\end{lemma}
Let us first explain the meaning of $H^1_\dR(S,2\pi\Z)$. Recall that there is a standard isomorphism $e:H^1_\dR(S,\R)\to\Hom(\pi_1(S),\R)$, which is defined by
$$e([\eta]):[\gamma]\mapsto\int_\gamma \eta~.$$
Then we denote $H^1_\dR(S,2\pi\Z)$ as the subgroup of $H^1_\dR(S,\R)$ such that 
$$H^1_\dR(S,2\pi\Z)=e^{-1}(\Hom(\pi_1(S),2\pi\Z))~.$$ That is, $H^1_\dR(S,2\pi\Z)$ consists of cohomology classes of 1-forms with periods which are integer multiples of $2\pi$.

\begin{proof}[Proof of Lemma \ref{lemma integer coefficients}]
Given $b$ and $\widehat b$, there exists a smooth function $\vartheta:S\to\mathbb S^1$ such that $\widehat b=R_\vartheta\circ b$. Consider the composition of vector space homomorphisms
\[
\xymatrix{
C^\infty(S,\mathbb S^1)\ar[r] &H^1_\dR(S,\R)\ar[r]^{\!\!\!\!\!\!\!\!\!e} &\Hom(\pi_1(S),\R)
}
\]
where the first arrow is
$$\vartheta\mapsto [d\vartheta]~,$$
as $d\vartheta$ is locally well-defined as a 1-form.

Let us observe that the composition 
$$F:C^\infty(S,\mathbb S^1)\to\Hom(\pi_1(S),\R)$$can be expressed as 
$$F(\vartheta)=2\pi \vartheta_*:\pi_1(S)\to \R~,$$
where $\vartheta_*$ is the map induced by $\vartheta$ from $\pi_1(S)$ to $\pi_1(\mathbb S^1)\cong \Z<\R$. Hence $F$
has image in $\Hom(\pi_1(S),2\pi\Z)$. 


Now, pick an oriented orthonormal frame $\{v_1,v_2\}$ for $h$. Then let $\widehat \omega'$ be the connection form of $\psi^*h'$ with respect to the frame
$$\{\widehat v_1',\widehat v_2'\}=\{\widehat b^{-1}v_1,\widehat b^{-1}v_2\}=\{b^{-1}R_{-\vartheta}v_1,b^{-1}R_{-\vartheta}v_2\}~,$$
which coincides with
$$\{R_{-\vartheta}'b^{-1}v_1,R_{-\vartheta}'b^{-1}v_2\}=\{R_{-\vartheta}'v_1',R_{-\vartheta}'v_2'\}~.$$ 

If $\omega'$ is the connection form of $\psi^*h'$ with respect to the frame
$\{v_1',v_2'\}$, by the same computation as in Lemma \ref{eta well defined}, we have $\widehat\omega'=\omega'-d\vartheta$. Therefore $\eta_{\psi,b}$ and $\eta_{\psi,b'}$ differ by the 1-form $d\vartheta$, which has integer coefficients as remarked above. 
\end{proof}

\begin{remark}
From the construction in the above proof, it follows that, given $b$ and $\widehat b$, $[\eta_{\psi,b}]=[\eta_{\psi,\widehat b}]$ in $H^1_\dR(S,\R)$ if and only if 
$$\widehat b=R_\theta\circ b$$
for $\theta:S\to\R$. That is, if and only if the function $\vartheta:S\to\mathbb S^1$ above can be lifted to a $\R$-valued map. 
\end{remark}

Lemma \ref{lemma integer coefficients} enables us to give the following definition:

\begin{defi} \label{defi boh}
Given two hyperbolic metrics $h$ and $h'$ on $S$ with area forms $\Omega_h$ and $\Omega_{h'}$, we define
$$\boh_{h,h'}:\symp_0(S,\Omega_h,\Omega_{h'})\to H^1_{\dR}(S,\R)/H^1_{\dR}(S,2\pi\Z)$$
by $\boh_{h,h'}(\psi):=[\eta_{\psi,b}]$, where $b$ is any smooth section of
 $\isom(TS,\psi^*h',h)$.
\end{defi}

\subsection*{Some properties of the 1-form $\eta_{\psi,b}$.}

We will now derive an equivalent expression for the previously defined 1-form $\eta_{\psi,b}$, which will be used later.

\begin{lemma} \label{lemma difference connections}
Let $h$ be a hyperbolic metric on $S$, let $b\in\Gamma^\infty(\mathrm{End}(TS))$ non-singular. If $\nabla$ is the Levi-Civita connection of $h$ and $\nabla'$ is the Levi-Civita connection of $h'=h(b\cdot,b\cdot)$, then:
\begin{equation}  \label{eq form twisted connection}
\nabla'_v w=b^{-1}\nabla_v b(w)+\alpha(v)J_{h'}(w)~,
\end{equation}
where $J_{h'}$ is the almost-complex structure induced by the metric $h'$ and $\alpha$ is the 1-form defined by
$$\alpha(v)=h(b^{-1}J_h (\star d^\nabla b),J_h v)~,$$
if $J_h$ is the almost-complex structure induced by $h$.
\end{lemma}

Recall that $d^\nabla b$ denotes the exterior derivative, which is the 2-form with values in $TS$ defined by:
\begin{equation} \label{eq exterior derivative}
d^\nabla b(v,w)=\nabla_v (b(w))-\nabla_w (b(v))-b[v,w]~,
\end{equation}
where $v$ and $w$ are any two vector fields. Furthermore, $\star d^\nabla b$ is the Hodge dual of the 2-form $d^\nabla b$, which can be defined as 
$$\star d^\nabla b=d^\nabla b(v_1,v_2)~,$$
where $\{v_1,v_2\}$ is an oriented orthonormal frame (and this expression does not depend on the chosen oriented orthonormal frame).

\begin{proof}[Proof of Lemma \ref{lemma difference connections}]
Observe that $b^{-1}\nabla b$ is a connection on $S$ compatible with the metric $h'$ --- but not symmetric in general. Hence the difference between $\nabla'$ and $b^{-1}\nabla b$ is a 1-form with values in the bundle of $h'$-skew-symmetric endomorphisms of $TS$. Hence we can write
\begin{equation*} 
\nabla'_v w=b^{-1}\nabla_v b(w)+\alpha(v)J_{h'}(w)~,
\end{equation*}
where $\alpha$ is a 1-form and 
\begin{equation} \label{eq complex structure twisted}
J_{h'}=b^{-1}J_h b
\end{equation}
 is the almost-complex structure induced by the metric $h'$. Then from the definitions, we have
\begin{align*}
d^\nabla b(v,w)&=\nabla_v b(w)-\nabla_w b(v)-b[v,w] \\
&=(b\nabla'_v w-\alpha(v)bJ_{h'}(w)) -(b\nabla'_w v-\alpha(w)bJ_{h'}(v))-b[v,w] \\
&=b(\nabla'_v w-\nabla'_w v-[v,w])-\alpha(v)J_h b(w)+\alpha(w)J_h b(v) \\
&=-\alpha(v)J_h b(w)+\alpha(w)J_h b(v)~,
\end{align*}
since we are assuming $\nabla'$ is torsion-free.

Let us choose an oriented orthonormal frame $\{v_1,v_2\}$ for $h$. Recall that the Hodge dual of the 2-form $d^\nabla b$ equals $\star d^\nabla b=d^\nabla b(v_1,v_2)$. Then 
\begin{align*}
\star d^\nabla b&=-\alpha(v_1)J_h bJ_h (v_1)-\alpha(v_2)J_h bJ_h (v_2) \\
&=-(J_h bJ_h )(\alpha(v_1)v_1+\alpha(v_2)v_2)~.
\end{align*}
Therefore, for every vector $v$,
\begin{align*}
\alpha(v)&=-h((J_h bJ_h )^{-1}(\star d^\nabla b),v) \\
&=-h(J_h b^{-1}J_h (\star d^\nabla b),v) \\
&=h(b^{-1}J_h (\star d^\nabla b),J_h v)~.
\end{align*}
This concludes the proof.
\end{proof}

Using Equation \eqref{eq form twisted connection}, we can now express the 1-form $\eta_{\psi,b}$ in the following way: if $\{v_1,v_2\}$ is an oriented orthonormal frame, then:
\begin{align*}
\eta_{\psi,b}(w)&=\omega'(w)-\omega(w) \\
&=h'(\nabla'_w b^{-1}v_1,b^{-1}v_2)-h(\nabla_w v_1,v_2) \\
&=h(b\nabla'_w b^{-1}v_1,v_2)-h(\nabla_w v_1,v_2) \\
&=h(\alpha(w)J_h v_1,v_2)=\alpha(w)~,
\end{align*}
since $J_h v_1=v_2$, and we have used Equation \eqref{eq complex structure twisted}. Therefore, using Lemma \ref{lemma difference connections}, we have the following:

\begin{cor} \label{cor altra espressione eta}
Given two hyperbolic metrics $h$ and $h'$ on $S$ with area forms $\Omega_h$ and $\Omega_{h'}$, a symplectomorphism $\psi\in\symp_0(S,\Omega_h,\Omega_h')$, and a smooth section $b$ of $\isom(TS,\psi^*h',h)$, 
\begin{equation} \label{eq altra espressione eta}
\eta_{\psi,b}=h(b^{-1}J_h(\star d^\nabla b),J_h\cdot)~,
\end{equation}
\end{cor}

Observe that, from Equation \eqref{eq altra espressione eta}, $\eta_{\psi,b}$ vanishes if and only if $\star d^\nabla b$ vanishes identically, which is equivalent to $d^\nabla b=0$. 
Hence from Corollary \ref{cor altra espressione eta}, we obtain:

\begin{cor} \label{cor vanishing codazzi condition}
Given two hyperbolic metrics $h$ and $h'$ on $S$ with area forms $\Omega_h$ and $\Omega_{h'}$, a symplectomorphism $\psi\in\symp_0(S,\Omega_h,\Omega_h')$, and a smooth section $b$ of $\isom(TS,\psi^*h',h)$, 
$$\eta_{\psi,b}=0\qquad\text{if and only if}\qquad
d^{\nabla}b=0~.$$
\end{cor}

\subsection*{The group structure}

Let us now study how the map $\boh_{h,h'}$ (defined in Definition \ref{defi boh}) transforms under composition of symplectomorphisms. In particular, this will prove that $\boh_{h,h}:\symp_0(S,\Omega_h)\to H^1_{\dR}(S,\R)/H^1_{\dR}(S,2\pi\Z)$ is a group homomorphism.

\begin{lemma} \label{composition formula}
Let $h$ and $h'$ be hyperbolic metrics on $S$. Then
for $\psi\in\symp_0(S,\Omega_h)$ and $\widehat\psi\in\symp_0(S,\Omega_h,\Omega_{h'})$,
$$\boh_{h,h'}(\widehat\psi\circ\psi)=\boh_{h,h}(\psi)+\boh_{h,h'}(\widehat\psi)~.$$
\end{lemma}
\begin{proof}
Let $b$ and $\widehat b$ be smooth sections of $\mathrm{End}(TS)$ such that
$\widehat\psi^*h'=h(\widehat b\cdot,\widehat b\cdot)$ and $\psi^*h=h(b\cdot,b\cdot)$. 
 Let us denote $\psi^*\widehat b=(d\psi)^{-1}\widehat b(d\psi)$. Then $\psi^*\widehat b\in\Gamma^\infty(\isom_\Sigma(\psi^*h,\psi^*\widehat\psi^*h'))$ and
$$(\widehat\psi \psi)^*h'=h((b\circ \psi^*\widehat b)\cdot,(b\circ \psi^*\widehat b)\cdot)\,.$$

To compute $[\eta_{\widehat\psi\psi}]$, let us pick an orthonormal frame $\{v_1,v_2\}$ for $h$, let $\omega$ be the associated connection form and $\omega''$ the connection form of $(\widehat\psi\psi)^*h'$ with respect to the frame $\{(b\circ \psi^*\widehat b)^{-1}v_1,(b\circ \psi^*\widehat b)^{-1}v_2\}$. Let $\omega'$ be the connection form of $\psi^*h$ for the frame $\{b^{-1}v_1,b^{-1}v_2\}$. Then
$$\eta_{\widehat\psi\psi,(b\circ \psi^*\widehat b)}=\omega''-\omega=(\omega''-\omega')+(\omega'-\omega)\,.$$
Now clearly $\omega'-\omega=\eta_{\psi,b}$. On the other hand, $\omega'$ and $\omega''$ are the pull-backs, via $\psi$, of the connection forms of $h$ and $\widehat\psi^*h'$, and thus $\omega''-\omega'=\psi^*\eta_{\widehat\psi,\widehat b}$. This shows that:
$$\eta_{\widehat\psi\psi,(b\circ \psi^*\widehat b)}=\eta_{\psi,b}+\psi^*\eta_{\widehat\psi,\widehat b}\,.$$
Taking cohomology classes, $[\psi^*\eta_{\widehat\psi,\widehat b}]=[\eta_{\widehat\psi,\widehat b}]$ since $\psi$ is isotopic to the identity, and therefore
$$\boh_{h,h'}(\widehat\psi\circ\psi)=\boh_{h,h}(\psi)+\boh_{h,h'}(\widehat\psi)\,,$$
which concludes the proof.
\end{proof}

In particular, if we choose $h=h'$, then Lemma \ref{composition formula} shows that, if $\psi,\widehat \psi\in\symp_0(S,\Omega_h)$,
$$\boh_{h,h}(\widehat\psi\circ\psi)=\boh_{h,h}(\psi)+\boh_{h,h'}(\widehat\psi)~.$$
Hence $\boh_{h,h}:\symp_0(S,\Omega_h)\to H^1_\dR(S,\R)/H^1_{\dR}(S,2\pi\Z)$ is a group homomorphism. Moreover, if $\psi$ varies smoothly in $\symp_0(S,\Omega_h)$, then the associated form $[\eta_{\psi,b}]$ varies smoothly. Hence $\boh_{h,h}$ is smooth, for the structure of infinite-dimensional Lie group of $\symp_0(S,\Omega_h)$. To summarize, we have:

\begin{cor} \label{lie group homo}
Given a hyperbolic metric $h$ on $S$, the map
$$\boh_{h,h}:\symp_0(S,\Omega_h)\to H^1_{\dR}(S,\R)/H^1_{\dR}(S,2\pi\Z)$$
is a Lie group homomorphism. 
\end{cor}

\subsection*{The coincidence in the quotient of $\flux$ and $\boh_{h,h}$}

In this subsection we restrict to the case in which the two metrics $h$ and $h'$ coincide. In this case, the map $\boh_{h,h}$ is a group homomorphism, as already observed. The main result of this part is the fact that the map $\boh_{h,h}$ coincides with the flux homomorphism, composed with the projection from $H^1_{\dR}(S,\R)$ to $H^1_{\dR}(S,\R)/H^1_{\dR}(S,2\pi\Z)$.

\begin{prop} \label{prop coincidence}
Given a hyperbolic metric $h$ on a closed oriented surface $S$, let $\Omega_h$ be the induced area form. Then the following diagram is commutative:
\begin{equation} \label{eq diagram commutative}
\xymatrix{
 \symp_0(S,\Omega_h) \ar[r]^-{\flux} \ar[dr]_-{\boh_{h,h}} & H^1_{\dR}(S,\R) \ar[d]^-\pi \\
 & H^1_{\dR}(S,\R)/H^1_{\dR}(S,2\pi\Z) \\
}
\end{equation}
\end{prop}

Surprisingly, Proposition \ref{prop coincidence} implies that the homomorphism $\boh_{h,h}$ (when the two hyperbolic metrics coincide) does not depend on the hyperbolic metric $h$, but only on its area form $\Omega_h$.

To prove Proposition \ref{prop coincidence}, we will first prove that the two maps coincide at the infinitesimal level. Recall that the Lie algebra $\mathfrak{symp}(S,\Omega_h)$ consists of smooth vector fields satisfying $\mathrm{div} X=\tr\nabla X=0$. The exponential map $\exp:\mathfrak{symp}(S,\Omega_h)\to\symp_0(S,\Omega_h)$ coincides with the flow of the vector field $X$, namely
$\exp(tX)=\psi_t\in\symp_0(S,\Omega_h)$
where 
$$\frac{d}{dt}\psi_t=X\circ \psi_t~.$$
Moreover, since $H^1_{\dR}(S,2\pi\Z)$ is discrete, the tangent space to the quotient $H^1_{\dR}(S,\R)/H^1_{\dR}(S,2\pi\Z)$ at the identity element is naturally identified with $H^1_{\dR}(S,\R)$ itself. With this convention:

\begin{lemma} \label{lemma diff boh}
Let $X\in\mathfrak{symp}(S,\Omega_h)$, and let $\psi_t\in\symp_0(S,\Omega_h)$ be the flow of $X$. Then 
$$\left.\frac{d}{dt}\right|_{t=0}\boh_{h,h}(\psi_t)=[\Omega_h(X,\cdot)]~.$$
\end{lemma}

\begin{proof}
In order to compute $\boh_{h,h}(\psi_t)$, recall that $\boh_{h,h}$ is defined as the class of the 1-form $\eta_{\psi_t,b_t}$, where we can choose a smoothly varying family $b_t\in\isom(TS,\psi_t^*h,h)$.

First of all, observe that, by differentiating the condition
$$h(d\psi_t\cdot,d\psi_t\cdot)=h(b_t\cdot,b_t\cdot)$$
we obtain that
$$h(\nabla_v X,w)+h(v,\nabla_w X)=h(\dot b(v),w)+h(v,\dot b(w))~,$$
where
$$\dot b=\left.\frac{d}{dt}\right|_{t=0}b_t$$
and as usual $\nabla$ is the Levi-Civita connection of $h$. 
In other words, the symmetric part of the operators $\nabla_X$ and $\dot b$ coincide. Since $\psi_t$ preserves the area, the generating field $X$ is divergence-free, that is $\tr\nabla X=0$. On the other hand, since $\det b_t= 1$, we have $\tr\dot b=0$. Therefore there exists a function $f:S\to\R$ such that
$$\nabla X=\dot b+fJ_h~,$$
where $J_h$ is the almost-complex structure determined by the metric $h$. Now, let us take exterior derivatives, applied to an othonormal frame $\{v_1,v_2\}$. Using that the curvature of $h$ is identically $-1$, one obtains
$$d^\nabla (\nabla X)(v_1,v_2)=J_h X~,$$
while on the other hand, using that $J_h$ is parallel (namely $\nabla J_h=0$), we have
$$d^\nabla(fJ_h)(v_1,v_2)=-df(v_1)v_1-df(v_2)v_2=-\grad f~.$$
Hence we have the following formula for the Hodge duals, which will be used later in the proof:
\begin{equation} \label{eq hodge duals}
\star d^\nabla \dot b=J_h X-\grad f~.
\end{equation}
Now from Corollary \ref{cor altra espressione eta}, we have:
$$\left.\frac{d}{dt}\right|_{t=0}\eta_{\psi_t,b_t}=\left.\frac{d}{dt}\right|_{t=0} h(b_t^{-1}J_h(\star d^\nabla b_t),J_h\cdot)$$
Recall that $b_0$ is the identity operator, hence $d^\nabla b_0=0$. Hence
\begin{align*}
\left.\frac{d}{dt}\right|_{t=0}\eta_{\psi_t,b_t}&=h(J_h(\star d^\nabla \dot b),J_h\cdot) \\
&=h(\star d^\nabla \dot b,\cdot) \\
&=h(J_h X-\grad f,\cdot) \\
&=\Omega_h(X,\cdot)+df~.
\end{align*} 
This shows that 
$$\left[\left.\frac{d}{dt}\right|_{t=0}\eta_{\psi_t,b_t}\right]=[\Omega_h(X,\cdot)]~,$$
and thus concludes the proof.
\end{proof}

The proof of Proposition \ref{prop coincidence} will be a consequence of the following lemma about (infinite-dimensional) Lie groups.

\begin{lemma} \label{lemma inf dim lie gps}
Let $G$ and $H$ be Lie groups, possibly of infinite dimension. Let $F_1,F_2:G\to H$ be Lie group homomorphism and let 
$(F_1)_*,(F_2)_*:\mathfrak g\to\mathfrak h$ be the induced Lie algebra homomorphisms. If $(F_1)_*=(F_2)_*$, then $F_1=F_2$.
\end{lemma}

Let us now conclude the proof of Proposition \ref{prop coincidence}.
\begin{proof}[Proof of Proposition \ref{prop coincidence}]
Observe that both $\pi\circ\flux$ and $\boh_{h,h}$ (the latter by Corollary \ref{lie group homo}) are Lie groups homomorphisms from $\symp_0(S,\Omega)$ to $H^1_\dR(S,\R)/H^1_\dR(S,2\pi\Z)$. In Lemma \ref{lemma diff boh}, we showed that the differential of $\boh_{h,h}$ is
$$(\boh_{h,h})_*(X)=[\Omega_h(X,\cdot)]\in H^1_{\dR}(S,\R)~,$$
for any $X\in\mathfrak{symp}(S,\Omega_h)$. On the other hand,
observe that $\psi_t=\exp(tX)$, where $X$ is the generating vector field of the 1-parameter subgroup $\psi_t$, and $\exp$ is the Lie group exponential map of $\symp_0(S,\Omega_h)$. Hence, by definition
$$\flux(\psi_t)=\int_0^1[\Omega_h(tX,\cdot)]ds=t[\Omega_h(X,\cdot)]~.$$
Therefore
$$(\flux)_*(X)=\left.\frac{d}{dt}\right|_{t=0}\flux(\psi_t)=[\Omega_h(X,\cdot)]~.$$
By appying Lemma \ref{lemma inf dim lie gps}, the proof follows.
\end{proof}

The following is a direct consequence of Proposition \ref{prop coincidence} and the fact that $H^1_{\dR}(S,2\pi\Z)$ is discrete in $H^1_{\dR}(S,\R)$:

\begin{cor} \label{cor codazzi hamiltonian}
Given a hyperbolic metric $h$ on the closed oriented surface $S$, the connected component of the identity in 
$\mathrm{Ker}(\boh_{h,h})$ coincides with $\ham(S,\Omega_h)$,
where $\Omega_h$ is the area form induced by $h$.
\end{cor}

\section{Application to Anti-de Sitter geometry} \label{sec preliminaries}

In this section, we will start by providing the necessary preliminary notions on Anti-de Sitter geometry, from the basic definitions concerning Anti-de Sitter space. Then we will show how a symplectomorphism of hyperbolic surfaces is associated to equivariant surfaces in Anti-de Sitter space, with the special case of minimal Lagrangian diffeomorphisms. We conclude with the proof of our main result and some remarks.

\subsection*{Anti-de Sitter space}
Let $\kappa$ be the Killing form  on the Lie group $\PSLR$, which is a bi-invariant bilinear form on the Lie algebra $\psl_2\R$ of signature $(2,1)$. The bilinear form $\kappa$ induces a Lorentzian metric on $\PSLR$, which we will denote by $g_\kappa$. Then Anti-de Sitter space of dimension 3 is defined as:
$$\AdS^3:=\left(\PSLR, \frac{1}{8}g_\kappa\right)~.$$
Therefore $\AdS^3$ is a Lorentzian manifold, homeomorphic to a solid torus, of constant sectional curvature. It turns out that, due to the normalization factor $1/8$, its sectional curvature is $-1$. 
From the construction, it follows that the identity component of the isometry group of $\AdS^3$ is:
$$\isom_0(\AdS^3)= \PSLR\times\PSLR~,$$
where $(\alpha,\beta)\in\PSLR\times\PSLR$ acts on $\AdS^3$ by
$$(\alpha,\beta)\cdot \gamma=\alpha\circ\gamma\circ \beta^{-1}~.$$
The group $\PSLR$ is isomorphic to the group of orientation-preserving isometries of $\Hyp^2$, the hyperbolic plane which we consider in the upper-half plane model:
$$\Hyp^2:=\left(\{z\in\C\,\mid\,\Im(z)>0\},\frac{|dz|}{\Im(z)}\right)~.$$
Then, using the structure of the isometry group, one can show the following fact, which will be of fundamental importance in this paper. We recall that a differentiable curve $\gamma:I\to\AdS^3$ is \emph{timelike} if $g_\kappa(\dot\gamma,\dot\gamma)<0$ at every point $\gamma(t)$, \emph{spacelike} if $g_\kappa(\dot\gamma,\dot\gamma)>0$, and \emph{lightlike} if $g_\kappa(\dot\gamma,\dot\gamma)=0$.
\begin{fact} \label{fact timelike geodesics}
There is a 1-1 correspondence
$$\{\text{timelike geodesics in }\AdS^3\}\leftrightarrow\Hyp^2\times\Hyp^2~,$$
which is defined by associating to $(x,y)\in\Hyp^2\times\Hyp^2$ the timelike geodesic 
\begin{equation} \label{eq closed geodesic}
L_{x,y}:=\{\gamma\in\PSLR\,\mid\,\gamma(y)=x\}~.
\end{equation}
This correspondence is equivariant with respect to the action of $\PSLR\times\PSLR$ on $\Hyp^2\times\Hyp^2$ by isometries of $\Hyp^2$ on each factor, and on the set of timelike geodesics induced by isometries of $\AdS^3$.
\end{fact}
Indeed, $L_{x_0,x_0}$ is a closed timelike geodesic for every $x\in\Hyp^2$, since it is a maximal compact subgroup of the Lie group $\PSLR$, hence the restriction of the bilinear form is negative definite. It is a geodesic since it is a 1-parameter group, and the Riemannian exponential map coincides with the Lie group exponential map for bi-invariant metrics as $g_\kappa$ is. 
It also turns out that the length of $L_{x_0,x_0}$ is $\pi$, and the arclength parameter is $1/2$ the angle of rotation of elliptic elements of $\PSLR$ fixing $x_0$.

The equivariance of such 1-1 correspondence can be easily checked, since:
$$(\alpha,\beta)\cdot L_{x,y}=L_{\alpha(x),\beta(y)}~.$$
This also shows that every timelike geodesic of $\AdS^3$ is of the form $L_{x,y}$ for some $x,y$, since it is the image of $L_{x_0,x_0}$ under some isometry $(\alpha,\beta)$ of $\AdS^3$.

There is a natural notion of boundary at infinity of Anti-de Sitter space, topologically a torus, which is defined in the following way:
$$\partial_\infty\AdS^3:=\RP^1\times \RP^1~,$$
where we declare that a sequence $\gamma_n\in\PSLR$ converges to a pair $(p,q)\in\RP^1\times \RP^1$ if and only if there exists a point $x\in\Hyp^2$ such that:
\begin{equation} \label{eq convergence boundary}
\begin{cases}
\gamma_n(x)\to p \\
\gamma_n^{-1}(x)\to q
\end{cases}~.
\end{equation}
It is not difficult to check that, if Equation \eqref{eq convergence boundary} holds for some point $x\in\Hyp^2$, then it holds for every other point $x'\in\Hyp^2$. It also turns out that every isometry of $\AdS^3$, of the form $(\alpha,\beta)\in\PSLR\times\PSLR$, extends to the boundary, with the obvious action of $(\alpha,\beta)$ on $\RP^1\times \RP^1$.

\subsection*{From invariant spacelike surfaces to symplectomorphisms of hyperbolic surfaces} \label{subsec embeddings sigma}
Recall that in this paper $S$ denotes a closed oriented hyperbolic surface of Euler characteristic $\chi(S)<0$. A smoothly embedded surface $\Sigma$ in $\AdS^3$ is \emph{spacelike} if its induced metric is a Riemannian metric. One of the main objects of this paper will be spacelike surfaces $\Sigma$ which are preserved by a representation
$$\rho:\pi_1(S)\to\PSLR\times\PSLR~,$$
which acts freely and properly discontinously on $\Sigma$. 
Following Mess, there is a precise description of these objects:
\begin{fact}[\cite{mess}] \label{fact mess 1}
Suppose $\Sigma$ is a smooth spacelike surface in $\AdS^3$ and 
$$\rho=(\rho_l,\rho_r):\pi_1(S)\to\PSLR\times\PSLR~,$$
is a representation which acts freely and properly discontinously on $\Sigma$. Then:
\begin{itemize}
\item The representations $\rho_l:\pi_1(S)\to\PSLR$ and $\rho_r:\pi_1(S)\to\PSLR$ are Fuchsian, that is, each of them acts freely and properly discontinously on $\Hyp^2$.
\item The frontier $\partial_\infty\Sigma$ of $\Sigma$ in $\partial_\infty\AdS^3\cong\RP^1\times\RP^1$ is the graph of the unique $\rho_l$-$\rho_r$-equivariant  homeomorphism of $\RP^1$.
\end{itemize}
\end{fact}

In general, given two representations $\rho_1$ and $\rho_2$, we say that a map $f$ is $\rho_1$-$\rho_2$-equivariant if, $f\circ \rho_1(\tau)=\rho_2(\tau)\circ f$ for every $\tau$. If $\rho_l$ and $\rho_r$ are Fuchsian, it is well-known that there exists a unique homeomorphism $f:\RP^1\to\RP^1$ which is $\rho_l$-$\rho_r$-equivariant, which is obtained as the value induced on $\RP^1=\partial_\infty\Hyp^2$ by the lift to $\Hyp^2$ of any diffeomorphism between the quotient closed surfaces $\Hyp^2/\rho_l(\pi_1(S))$ and $\Hyp^2/\rho_r(\pi_1(S))$ isotopic to the identity.

Hence, let $\rho_l$ and $\rho_r$ be Fuchsian representations, so that the oriented topological surface $S$ is homeomorphic to the hyperbolic surfaces $\Hyp^2/\rho_l(\pi_1(S))$ and $\Hyp^2/\rho_r(\pi_1(S))$. We will denote by $h$ and $h'$ the hyperbolic metrics induced on $S$ in this way. We will now define the symplectomorphism
$$\phi_\Sigma:(S,\Omega_{h_l})\to(S,\Omega_{h_r})~,$$
where $\Omega_{h_l}$ and $\Omega_{h_r}$ are the area forms induced by the metrics $h_l$ and $h_r$ and by the orientation of $S$.
For this purpose, consider the map 
\begin{equation} \label{eq defi iota}
\iota_\Sigma:\Sigma\to\Hyp^2\times\Hyp^2~,
\end{equation}
which maps $\gamma\in\Sigma$ to the pair $(x,y)$ which represents the  unique timelike geodesic orthogonal to $\Sigma$ at $\gamma$, using Fact \ref{fact timelike geodesics}. Let $\pi_l,\pi_r:\Hyp^2\times\Hyp^2\to\Hyp^2$ denote the projections on the first and second factor. Then we have:

\begin{fact}[\cite{Schlenker-Krasnov},\cite{bon_schl},\cite{ksurfaces}] \label{fact projections}
Given two Fuchsian representations $\rho_l$ and $\rho_r$ and any smooth spacelike surface $\Sigma$ in $\AdS^3$ invariant by the representation 
$$\rho=(\rho_l,\rho_r):\pi_1(S)\to \PSLR\times\PSLR~,$$
\begin{itemize}
\item For $\gamma\in\Sigma$, if the curvature of the induced metric on $\Sigma$ is different from zero at $\gamma$, then $\pi_l\circ\iota_\Sigma$ and $\pi_r\circ\iota_\Sigma$ are local diffeomorphisms in a neighborhood of $\gamma$.
\item If the curvature of the induced metric on $\Sigma$ is everywhere different from zero, then $\pi_l\circ\iota_\Sigma$ and $\pi_r\circ\iota_\Sigma$ are global diffeomorphisms which extend to homeomorphisms from $\partial_\infty\Sigma$ to $\RP^1$.
\item In this case, the composition 
$$\widetilde \phi_\Sigma=(\pi_r\circ\iota_\Sigma)\circ(\pi_l\circ\iota_\Sigma)^{-1}$$
 induces a symplectomorphism
$$\phi_\Sigma:(S,\Omega_{h_l})\to(S,\Omega_{h_r})$$
isotopic to the identity, where $\Omega_{h_l}$ and $\Omega_{h_r}$ are the area forms of $h_l$ and $h_r$.
\end{itemize}
\end{fact}

In the second point, we remark that, since $\Sigma$ is invariant by $\rho(\pi_1(S))$ with quotient homeomorphic to the closed surface $S$, with $\chi(S)<0$, by the Gauss-Bonnet formula the curvature of the induced metric is always different from zero if and only if it is negative everywhere.

For the last point, it follows from the definitions that $\pi_l\circ\iota_\Sigma$ is $\rho$-$\rho_l$-equivariant, and $\pi_r\circ\iota_\Sigma$ is $\rho$-$\rho_r$-equivariant, hence the composition $(\pi_r\circ\iota_\Sigma)\circ(\pi_l\circ\iota_\Sigma)^{-1}$ induces a map in the quotient from $(S,h_l)=\Hyp^2/\rho_l(\pi_1(S))$ to $(S,h_r)=\Hyp^2/\rho_r(\pi_1(S))$. 

In \cite{bon_schl} and \cite{ksurfaces}, some more precise properties of the map $\phi_\Sigma$ were given. That is:

\begin{fact} \label{fact tensor b}
Given two Fuchsian representations $\rho_l$ and $\rho_r$ and any smooth spacelike surface $\Sigma$ in $\AdS^3$ invariant by the representation 
$$\rho=(\rho_l,\rho_r):\pi_1(S)\to \PSLR\times\PSLR~,$$
let $\phi_\Sigma:(S,\Omega_{h_l})\to(S,\Omega_{h_r})$ be the associated symplectomorphism.
Then there exists a smooth section $b\in\Gamma^\infty(\mathrm{End}(TS))$ such that:
\begin{enumerate}
\item $\phi_\Sigma^*h_r=h_l(b\cdot,b\cdot)$;
\item $\det b=1$;
\item $d^{\nabla_l}b=0$, where $\nabla_l$ is the Levi-Civita connection of the metric $h_l$;
\item $\tr b\neq -2$.
\end{enumerate}
\end{fact}

We recall again that, for a connection $\nabla$, $d^\nabla b$ is a 2-form with values in $TS$, defined in Equation \eqref{eq exterior derivative}.

\begin{remark}
Let us make some comments on these conditions. Condition $(1)$ is equivalent to saying that $b$ is a smooth section of the subbundle $\isom(TS,\phi_\Sigma^*h_r,h_l)$, whose fiber over the point $x\in S$ consists of the linear isometries $b_x:(T_x S,\phi_\Sigma^*h_r)\to(T_x S,h_l)$. 

First, such a section $b$ is not uniquely determined. In fact, if $b$ satisfies the conditions $(1)$ and $(2)$, then $R\circ b$ still satisfies $(1)$ and $(2)$, where $R\in\Gamma^\infty(\mathrm{Isom}(TS,h_l,h_l))$ --- that is, the tensor $b$ can be post-composed with an isometry of the metric $h_l$ at any point. Conditions $(1)$ and $(2)$ imply that $\phi_\Sigma$ preserves the area forms of $h_l$ and $h_r$, hence $\phi_\Sigma:(S,\Omega_{h_l})\to(S,\Omega_{h_r})$ is a symplectomorphism.

Further, if $b$ satisfies $(1)$, $(2)$ and $(3)$, then for any number $\theta_0\in (0,2\pi)$ also the section $R_{\theta_0}\circ b$ satisfies $(1)$, $(2)$ and $(3)$, where 
$$R_{\theta_0}=(\cos\theta_0) \mathrm{id}+(\sin\theta_0) J_{h_l}\in\Gamma^\infty(\mathrm{Isom}(TS,h_l,h_l))$$
is the section which at every point $x\in(S,h_l)$ rotates $T_x S$ counterclockwise of an angle $\theta_0$, for the metric $h_l$. (Here $J_{h_l}$ is the almost-complex structure induced by the metric $h_l$.) Finally, as condition $(4)$ is an open condition, if $b$ satisfies $(1)$, $(2)$, $(3)$ and $(4)$, then $R_{\theta_0}\circ b$ still satisfies $(1)$, $(2)$, $(3)$ and $(4)$ for small $\theta_0$.
\end{remark}

Finally, we sketch an argument explaining why Fact \ref{fact tensor b} is true. On the metric universal cover $\Hyp^2$ of $(S,h_l)$, a particular section $\widetilde b\in\Gamma^\infty(\mathrm{End}(T\Hyp^2))$ satisfying the above properties can be constructed explicitly as
\begin{equation} \label{eq tilde b}
\widetilde b=(\mathrm{id}+J_\Sigma B_\Sigma)^{-1}\circ (\mathrm{id}-J_\Sigma B_\Sigma)~,
\end{equation}
where $B_\Sigma$ is the shape operator, and $J_\Sigma$ is the almost-complex structure of the first fundamental form, both computed with respect to the embedding 
$$\sigma:=(\pi_l\circ\iota_\Sigma)^{-1}:\Hyp^2\to\AdS^3~.$$
 It can be checked that $\widetilde b$ is $\rho_l$-invariant, and thus it defines a section $b\in\Gamma^\infty(\mathrm{End}(TS))$, which always satisfies 
$(1)$, $(2)$, $(3)$ and $(4)$. See \cite[Proposition 4.9]{ksurfaces} for more details.

\subsection*{Minimal Lagrangian diffeomorphisms}
A particular example of the above correspondence occurs when $\Sigma$ is a \emph{maximal} surface, that is, its shape operator $B_\Sigma$ satisfies $\tr B_\Sigma\equiv 0$. In this case, the associated map $\phi_\Sigma:(S,h_l)\to(S,h_r)$ will be \emph{minimal Lagrangian}, namely, it is a symplectomorphism for the induced area forms $\Omega_{h_l}$ and $\Omega_{h_r}$, and the graph of $\phi_\Sigma$ is a minimal surface in $(S\times S,h_l\oplus h_r)$. Hence this is a notion which depends on the hyperbolic metrics $h_l$ and $h_r$, not only on the symplectic form.  The term \emph{Lagrangian} comes from the fact that, if $\phi_\Sigma$ is a symplectomorphism, then the graph of $\phi_\Sigma$
is a Lagrangian surface in the symplectic manifold $(S\times S,\pi_l^*\Omega_{h_l}-\pi_r^*\Omega_{h_r})$.

Recall that an operator $b\in\Gamma^\infty(\mathrm{End}(TS))$ is $h$-self-adjoint, for $h$ a Riemannian metric on $S$, if
$$h(b\cdot,\cdot)=h(\cdot,b\cdot)~.$$
The correspondence between maximal surfaces and minimal Lagrangian diffeomorphisms is consequence of the 
following characterization (see \cite{labourieCP}):

\begin{lemma} \label{lemma minimal lag}
Let $h_l$ and $h_r$ be hyperbolic metrics on $S$. Then $\phi_{\ML}:(S,h_l)\to(S,h_r)$ is minimal Lagrangian if and only if there exists a smooth $h_l$-self-adjoint section $b_\lab\in\Gamma^\infty(\mathrm{End}(TS))$ such that:
\begin{enumerate}
\item $\phi_\Sigma^*h_r=h_l(b_\lab\cdot,b_\lab\cdot)$;
\item $\det b_\lab=1$;
\item $d^{\nabla_l}b_\lab=0$, where $\nabla_l$ is the Levi-Civita connection of the metric $h_l$.
\end{enumerate}
In this case, the tensor $b_\lab$ (which is unique) is called \emph{Labourie operator}. 
\end{lemma}

Hence we have the following:

\begin{fact} \label{fact maximal min lag}
Given two Fuchsian representations $\rho_l$ and $\rho_r$, if $\Sigma_0$ is a maximal surface in $\AdS^3$ invariant by the representation 
$$\rho=(\rho_l,\rho_r):\pi_1(S)\to \PSLR\times\PSLR~,$$
then the associated diffeomorphism $\phi_{\Sigma_0}:(S,h_l)\to(S,h_r)$ is minimal Lagrangian. 
\end{fact} 
The proof of Fact \ref{fact maximal min lag} follows from checking that the tensor $b$ defined in Equation \eqref{eq tilde b} is $h_l$-self-adjoint if $\Sigma_0$ is a maximal surface, that is, if $\tr B_{\Sigma_0}\equiv 0$. See \cite{bon_schl}.

The following is a theorem of existence and uniqueness of minimal Lagrangian maps isotopic to the identity between any two compact oriented hyperbolic surfaces. It was proved in \cite{labourieCP} and \cite{Schoenharmonic}, and it can also be inferred from the existence and uniqueness of the invariant maximal surface $\Sigma_0$, see \cite{bbzads} and \cite{bon_schl}.

\begin{theorem} \label{thm minimal lag}
Given any two hyperbolic metrics $h,h'$ on the compact oriented surface $S$, there exists a unique minimal Lagrangian diffeomorphism $\phi_{\ML}:(S,h)\to(S,h')$ isotopic to the identity.
\end{theorem}

\subsection*{Proof of the main theorem}

Let us now prove the main result of this paper, which was already stated in the introduction. As usual, given two Fuchsian representations $\rho_l,\rho_r:\pi_1(S)\to\PSLR$, we denote $(S,h_l)=\Hyp^2/\rho_l(\pi_1(S))$ and similarly $(S,h_r)=\Hyp^2/\rho_r(\pi_1(S))$, and $\Omega_{h_l}$ and 
$\Omega_{h_r}$ are the induced area forms.

\begin{theorem} \label{main thm}
Let $\rho_l,\rho_r:\pi_1(S)\to\PSLR$ be Fuchsian representations and let $\Sigma\subset \AdS^3$ be a smooth, embedded spacelike surface invariant for the representation $\rho=(\rho_l,\rho_r):\pi_1(S)\to\isom_0(\AdS^3)$, whose curvature is negative. 
Then 
$$\phi_\Sigma=\phi_{\ML}\circ \psi~,$$
where 
\begin{itemize}
\item $\phi_\Sigma:(S,\Omega_{h_l})\to (S,\Omega_{h_r})$ is the diffeomorphism associated to $\Sigma$;
\item  $\phi_{\ML}:(S,{h_l})\to (S,{h_r})$ is the unique minimal Lagrangian diffeomorphism isotopic to the identity;
\item $\psi$ is a Hamiltonian symplectomorphism, for the area form $\Omega_{h_l}$.
\end{itemize}
\end{theorem}

\begin{proof}
Let $\psi=\phi_{\ML}^{-1}\circ\phi_{\Sigma}$. Let $b$ the section of $\isom(TS,\phi_\Sigma^*h_r,h_l)$ given by Fact \ref{fact tensor b}, and $b_\lab$ the analogous section for $\phi_\ML$, which was also introduced in Lemma \ref{lemma minimal lag}. Using Lemma \ref{composition formula}, we have
$$\boh_{h,h'}(\phi_{\Sigma})=\boh_{h,h'}(\phi_\ML)+\boh_{h,h}(\psi)~.$$
Now, since $d^{\nabla_l}b=d^{\nabla_l}b_\lab=0$ by Fact \ref{fact tensor b} (and Lemma \ref{lemma minimal lag}), using Corollary \ref{cor vanishing codazzi condition} we have $\eta_{\phi_{\Sigma},b}=\eta_{\phi_{\ML},b_\lab}=0$, hence $\boh_{h,h'}(\phi_{\Sigma})=\boh_{h,h'}(\phi_\ML)=0$ and therefore $\boh_{h,h}(\psi)=0$. From Definition \ref{defi boh} and Proposition \ref{eq diagram commutative}, this implies that $\flux(\psi)\in H^1_{\dR}(S,2\pi\Z)$.

On the other hand, consider the maximal surface $\Sigma_0$ invariant by the representation $\rho$ (as in Fact \ref{fact maximal min lag}).  By the results in \cite{mess}, the surfaces $\Sigma$ and $\Sigma_0$ induce closed embedded surfaces in a (so-called maximal globally hyperbolic) three-manifold $M$ homeomorphic to $S\times\R$. Let us call $\overline \Sigma$ and $\overline \Sigma_0$ be the surfaces induced in the quotient. Both $\overline \Sigma$ and $\overline \Sigma_0$ are homotopic to $S\times\{\star\}$. Hence one can find a smooth isotopy $f_t:S\to M$ which is a homotopy equivalence at every time $t$, such that $f_0(S)=\overline \Sigma_0$, and $f_1(S)=\overline \Sigma$, and $f_t(S)$ is a smooth spacelike surface for every $t$. Lifting to $\AdS^3$, the surfaces  $\Sigma_t$ which project to $f_t(S)$ induce maps $\phi_{\Sigma_t}$ which vary smoothly. If we put $\phi_{\Sigma_t}=\phi_{\ML}\circ \psi_t$, the corresponding classes $\flux(\psi_t)$ vary smoothly, and are all in $H^1_{\dR}(S,2\pi\Z)$ by the same argument as the first part of this proof. Moreover, $\psi_0$ is the identity, and therefore $\flux(\psi_0)=0$. Since $H^1_{\dR}(S,2\pi\Z)$ is discrete, it then follows that $\flux(\psi_1)=\flux(\psi)=0$. Hence by Theorem \ref{thm flux hamiltonian}, $\psi$ is Hamiltonian. This concludes the proof.
\end{proof}

Let us remark that the second part is exactly in the spirit of Corollary \ref{cor codazzi hamiltonian}, by using the fact that there is a smooth interpolation between $\phi_\Sigma$ and $\phi_\ML$.
By specializing to the case in which the two hyperbolic metrics coincide, we get the following:

\begin{cor}
Given a Fuchsian representation $\rho_0:\pi_1(S)\to\PSLR$, let  
$$\phi_\Sigma:\Hyp^2/\rho_0(\pi_1(S))\to \Hyp^2/\rho_0(\pi_1(S))$$ be a symplectomorphism induced by a smooth embedded spacelike surface $\Sigma\subset\AdS^3$ invariant by $(\rho_0,\rho_0)(\pi_1(S))<\PSLR\times\PSLR$. Then $\phi_\Sigma$ is Hamiltonian.
\end{cor}

\subsection*{Reconstructing smooth spacelike surfaces}

We conclude by a discussion about the question whether the condition that $\phi=\phi_{\ML}\circ \psi$ for a symplectomorphism $\phi:(S,h_l)\to (S,h_r)$, where $\phi_\ML$ is minimal Lagrangian and $\psi$ is Hamiltonian, is also sufficient to obtain $\phi$ as the symplectomorphism associated to an invariant surface in $\AdS^3$. In fact, by Theorem \ref{main thm} this is a necessary condition.

Using the tools from \cite{ksurfaces}, given a symplectomorphism $\phi:(S,\Omega_l)\to (S,\Omega_r)$ and any section $b\in\Gamma(\isom(TS,\phi^*h_r,h_l))$, a map $\sigma_{\phi,b}:\Hyp^2\to\AdS^3$ can be reconstructed explicitly, such that $\sigma_{\phi,b}$ is equivariant for the action of $\rho_l:\pi_1(S)\to \PSL(2,\R)$ on $\Hyp^2$ and of the pair $\rho=(\rho_l,\rho_r)$ on $\AdS^3$. 

The map $\sigma_{\phi,b}:\Hyp^2\to\AdS^3$ is defined by the condition that $\sigma_{\phi,b}(x)$ is the unique isometry $\sigma\in\PSL(2,\R)$ (recalling that $\AdS^3$ is the Lie group $\PSL(2,\R)$) such that
\begin{equation} \label{definition1}
\begin{cases} \sigma(\phi(x))=x \\
d\sigma_{\phi(x)}\circ d\phi_x=-b_x 
\end{cases}~.
\end{equation}
Moreover, such map is orthogonal to the foliation in timelike lines of the form $L_{x,\phi(x)}$ provided $b$ satisfies $d^{\nabla_l}b=0$ (\cite[Corollary 5.7]{ksurfaces}), but in general it is not an immersion. 

Therefore, given a section $b\in\Gamma(\isom(TS,\phi^*h_r,h_l))$, suppose $[\eta_{\phi,b}]=0\in H^1_\dR(S,\R)$. (This assumption is satisfied if $\phi=\phi_\ML\circ \psi$, for $\phi_\ML:(S,h_l)\to(S,h_r)$ the minimal Lagrangian diffeomorphism and $\psi:(S,\Omega_l)\to(S,\Omega_l)$ Hamiltonian, by an argument similar to above.)
 Then by the same argument of the proof of Lemma \ref{lemma integer coefficients}, one can find a function $\theta:S\to\R$ such that $\eta_{\phi,R_\theta b}=0$, which is equivalent to $d^{\nabla_l}(R_\theta b)=0$ by Corollary \ref{cor vanishing codazzi condition}. If $\sigma_{\phi,R_\theta b}$ is an embedding, then the image $\Sigma=\sigma_{\phi,R_\theta b}(\Hyp^2)$ is a smoothly embedded spacelike surface, invariant by the representation $\rho=(\rho_l,\rho_r)$, whose associated map is $\phi$. 
 
However, the map $\sigma_{\phi,R_\theta b}$ can be singular. Observe also that the function $\theta$ is not uniquely determined, but there are possible choices which differ by adding a constant. In general, $\sigma_{\phi,R_\theta b}$ might be non-singular only for some
of these choices of $\theta$ (i.e. up to adding a constant function). 

In \cite{bonsepequivariant}, the authors showed that such a map $\sigma_{\phi,R_\theta b}$ always lifts to an embedding into the future timelike unit tangent bundle of the universal cover $\widetilde{\AdS^3}$. Finally, let us remark that if one picks a section $b\in\Gamma(\isom(TS,\phi^*h_r,h_l))$ such that $[\eta_{\phi,\widehat b}]\in H^1_\dR(S,2\pi\Z)$, then one can still define a map $\sigma_{\phi, \widehat b}$ as in Equation \eqref{definition1}. However, $\sigma_{\phi, b}$ and $\sigma_{\phi, \widehat b}$ will lift to maps from $\Hyp^2\to\widetilde{\AdS^3}$ which are equivariant for \emph{different} lifts to $\isom(\widetilde{\AdS^3})$ of the representation $\rho:\pi_1(S)\to\isom(\AdS^3)$. 

Lemma 2.1 in \cite{bonsepequivariant} shows that all the $\rho_l$-$\rho$-equivariant embeddings into $\AdS^3$ lift to embeddings into $\widetilde{\AdS^3}$ which are equivariant for the \emph{same} lift of $\rho$. This implies that the map $\sigma_{\phi, \widehat b}$ will never be an embedding if $[\eta_{\phi,\widehat b}]\in H^1_\dR(S,2\pi\Z)$ but $[\eta_{\phi,\widehat b}]\neq 0$.


\bibliographystyle{alpha}
\bibliography{../bs-bibliography}

\end{document}